\documentclass[psamsfonts]{amsart}

\usepackage{amssymb,amsfonts,amscd}
\usepackage[all,arc]{xy}
\usepackage[margin=1.2in]{geometry} 
\usepackage{enumerate, bbm}
\usepackage{mathrsfs}
\usepackage{mathtools}
\usepackage{hyperref}
\usepackage[usenames, dvipsnames]{color}
\usepackage{graphicx}
\usepackage{enumitem} 
\graphicspath{ {TexImages/} }

\newtheorem{thm}{Theorem}[section]
\newtheorem{cor}[thm]{Corollary}
\newtheorem{prop}[thm]{Proposition}
\newtheorem{lem}[thm]{Lemma}
\theoremstyle{definition}

\newtheorem{conj}[thm]{Conjecture}\theoremstyle{remark}

\setenumerate[1]{label=\thesection.\arabic*.} 
\setenumerate[2]{label*=\arabic*.} 
\setlist[enumerate]{itemsep=2ex, topsep=2ex} 
\setlist[itemize]{itemsep=2ex, topsep=2ex}

\renewcommand{\l}{\left}
\renewcommand{\r}{\right}

\newcommand{\half}{\frac{1}{2}}

\newcommand{\sig}{\sigma}

\newcommand{\pa}{\partial}

\newcommand{\tr}[1]{\textrm{#1}}
\newcommand{\rec}[1]{\frac{1}{#1}}
\newcommand{\f}[2]{\frac{#1}{#2}}

\newcommand{\sub}{\subseteq}

\newcommand{\floor}[1]{\left\lfloor #1\right\rfloor}

\newcommand{\cD}{\mr{des}'}
\newcommand{\cA}{\mr{asc}'}
\newcommand{\A}{\mr{asc}}
\newcommand{\D}{\mr{des}}
\newcommand{\mr}[1]{\mathrm{#1}}

\newcommand{\mf}[1]{\bar{#1}}
\newcommand{\UD}{\genfrac{\lbrace}{\rbrace}{0pt}{}}

\newcommand{\alt}[1]{#1}

\hypersetup{
	colorlinks,
	citecolor=blue,
	filecolor=blue,
	linkcolor=blue,
	urlcolor=blue,
	linktocpage
}

\title{Ballot Permutations and Odd Order Permutations}

\author{Sam Spiro}
\date{\today}
\begin{document}
\begin{abstract}
 	A permutation $\pi$ is ballot if, for all $k$, the word $\pi_1\cdots \pi_k$ has at least as many ascents as it has descents.  Let $b(n)$ denote the number of ballot permutations  of order $n$, and let $p(n)$ denote the number of permutations which have odd order in the symmetric group $S_n$.  Callan conjectured that $b(n)=p(n)$ for all $n$, which was proved by Bernardi, Duplantier, and Nadeau.
	
	We propose a refinement of Callan's original conjecture.  Let $b(n,d)$ denote the number of ballot permutations with $d$ descents.  Let $p(n,d)$ denote the number of odd order permutations with $M(\pi)=d$, where $M(\pi)$ is a certain statistic related to the cyclic descents of $\pi$.  We conjecture that $b(n,d)=p(n,d)$ for all $n$ and $d$. We prove this stronger conjecture for the cases $d=1,\ 2,\ 3$, and $d=\lfloor(n-1)/2\rfloor$, and in each of these cases we establish formulas for $b(n,d)$ involving Eulerian numbers and Eulerian-Catalan numbers.
\end{abstract}
	\maketitle
\section{Introduction}
Given a word $w=w_1w_2\cdots w_n$ whose letters are positive integers, we define the up-down signature $q^w=(q_1^w,q_2^w,\cdots, q_{n-1}^w)$ to be such that
\[
	q_i^w=\begin{cases*}
	+1 & $\pi_i\le\pi_{i+1}$,\\ 
	-1 & $\pi_{i}>\pi_{i+1}$.
	\end{cases*}
\]
We say that $i$ is an ascent of $w$ if $q_i^w=+1$ and that $i$ is a descent of $w$ if $q_i^w=-1$.  For example, if $w=31452$, then $q^w=(-1,+1,+1,-1)$, its ascents are 2 and 3, and its descents are 1 and 4.  We let $\A(w)$ and $\D(w)$ denote the number of ascents and descents of the word $w$, respectively.  One can encode the up-down signature of a word as a binary string.  To this end, we define the index of a word $w$ to be the binary string $r^w$ of length $|w|-1$ satisfying $(-1)^{r_i^w}=q_i^w$ for all $i$.  For example, the index of $w=31452$ is 1001.  These two concepts are equivalent ot one another, but depending on the circumstances one is often notionally more convenient to use than the other.

Let $S_n$ denote the group of permutations of size $n$.  The problem of enumerating the number of permutations in $S_n$ with a given up-down signature started with Andr{\'e} \cite{Andre} who deduced the exponential generating function for the number of permutations $\pi$ with up-down signature of the form $q^\pi=(+1,-1,+1,-1,\ldots)$.  This work was generalized by Niven \cite{Niven} who provided a formula for the number of $\pi \in S_n$ such that $q^\pi=q$ for any fixed up-down signature $q$.  More recent results related to up-down signatures include work by Brown, Fink, and Willbrand \cite{Brown} and Shevelev and Spilker \cite{Spilker}.  Of particular interest to us is some recent work of Shevelev. Given any binary string $r$, let $\UD{n}{r}$ denote the number of permutations $\pi\in S_n$ which have index $r^\pi=r00\cdots 0$.  Shevelev \cite{Shevelev} showed that $\UD{n}{r}$ is a polynomial in $n$ for any fixed $r$ with $|r|\le n+1$, and moreover provided explicit formulas for these polynomials for any given choice of $r$.

In this paper we are interested in permutations whose up-down signatures satisfy a certain property.  We will say that a permutation $\pi$ is ballot if $\sum_1^k \alt{q_i^\pi}\ge 0$ for all $1\le k\le n-1$.  Equivalently, a permutation $\pi$ is ballot if $\pi_1\cdots \pi_k$ has at least as many ascent as descents for all $k$.  For example, $\pi=31452$ is not ballot since $\sum_1^1 \alt{q_i^\pi}=-1$, but one can verify that $\sig=14352$ is ballot.  We let $B(n)$ denote the set of all ballot permutations of size $n$, and we let $b(n)=|B(n)|$.

We will say that a permutation $\pi$ is an odd order permutation (abbreviated OOP) if the order of $\pi$ is odd in $S_n$, which is equivalent to $\pi$ being the product of only odd cycles.  For example, $\pi=(3,1,4)(2,5,6,7,9)$ is an OOP since it has order 15 in $S_9$.  We let $P(n)$ denote the set of OOP's of size $n$, and we let $p(n)=|P(n)|$.

Callan \cite{Callan} conjectured that ballot permutations and OOP's are equinumerous, and this was proven by Bernardi, Duplantier, and Nadeau.
\begin{thm}\label{T-Cal}\cite{Nadeau}
	For all $n$,
	\[b(n)=p(n)=\begin{cases*}
	[(n-1)!!]^2 & $n$ even,\\ 
	n!!(n-2)!! & $n$ odd,
	\end{cases*}\]
	where $(2m-1)!!:=(2m-1)\cdot (2m-3)\cdot \cdots3\cdot 1$
\end{thm}
Based on experimental data, we believe that a refined version of Theorem~\ref{T-Cal} is true.  Let $B(n,d)$ denote the set of permutations of $B(n)$ with exactly $d$ descents, and let $b(n,d)=|B(n,d)|$.  We note that $B(n,d)=\emptyset$ whenever $d>\floor{(n-1)/2}$ since any $\pi\in B(n,d)$ would have $\sum_1^{n-1} \alt{q_i^\pi}<0$, and hence $\pi$ would not be ballot.

We wish to define an analog for the descent statistic in the context of OOP's.  Given a cycle $\mf{c}=(c_1,\ldots,c_k)$ of a permutation $\pi$, we let $\cA(\mf{c})$ denote the number of cyclic ascents of $\mf{c}$.  That is, $\cA(\mf{c})$ is the number of ascents in the word $c_1c_2\cdots c_kc_1$.  We similarly define $\cD(\mf{c})$ to be the number of cyclic descents of $\mf{c}$.  We let $M(\mf{c})=\min(\cA(\mf{c}),\cD(\mf{c}))$.  For example, if $\mf{c}=(4,2,8,5,6)$ we have $\cA(\mf{c})=2,\ \cD(\mf{c})=3$, and hence $M(\mf{c})=2$.  For a permutation $\pi=\mf{c}_1\mf{c}_2\cdots \mf{c}_k$ written in cycle notation, we define $M(\pi)=\sum_1^k M(\mf{c}_i)$.  For example, if $\pi=(1,3,9)(4,2,8,5,6)(7)$, then $M(\pi)=1+2+0=3$.  Let $P(n,d)$ denote the set of permutations of $P(n)$ with $M(\pi)=d$, and let $p(n,d)=|P(n,d)|$.

\begin{conj}\label{C-Main}
	$b(n,d)=p(n,d)$ for all $n$ and $d$.
\end{conj}
Conjecture~\ref{C-Main} is trivially true for $d=0$.  We show that it is also true for $d=1,\ 2$, and 3.  In order to state our results, we define the Eulerian number $E(n,d)$ to be the number of permutations of size $n$ with exactly $d$ descents.  We adopt the convention that $E(0,d)=0$ for $d>1$ and $E(0,0)=1$.  We note that one can show the following \cite{Wolf}.
\begin{align}
E(n,1)&=2^n-n-1,\label{E1}\\ 
E(n,2)&=3^n-(n+1)2^n+{n+1\choose 2},\label{E2}\\ 
E(n,3)&=4^n-(n+1)3^n+{n+1\choose 2}2^{n}-{n+1\choose 3}.\label{E3}
\end{align}

\begin{thm}\label{T-1}
	For all $n\ge 1$, \[b(n,1)=p(n,1)=2E(n-1,1).\]
	Moreover, there exists an explicit bijection between $B(n,1)$ and $P(n,1)$.
\end{thm}
\begin{thm}\label{T-2}
	For all $n\ge 2$,
	\[b(n,2)=p(n,2)=3E(n-1,2)-2{n\choose 3}+{n\choose 2}-1.\]
\end{thm}

\begin{thm}\label{T-3}
	For all $n\ge 4$,
	\[
		p(n,3)=b(n,3)=4E(n-1,3)-\l({n\choose 3}-{n\choose 2}+4\r)2^{n-2}-22{n\choose 5}+16 {n\choose 4}-4{n\choose 3}+2n.
	\]
\end{thm}

The ubiquity of $(d+1)E(n-1,d)$ in the above three formulas is not a coincidence.  Indeed, we will show the following.

\begin{prop}~\label{P-AscStart}
	Let $A(n,d)$ denote the set of permutations of size $n$ with $d$ descents and which begin with an ascent.  Let $a(n,d)=|A(n,d)|$.  Then
	\[
		a(n,k)=(d+1)E(n-1,d).
	\]
\end{prop}
From this result one can quickly obtain the formulas for $b(n,1)$ and $b(n,2)$.  Indeed, we always have $B(n,d)\sub A(n,d)$, and when $d=1$ this is an equality, giving the formula for $b(n,1)$ in Theorem~\ref{T-1}.  The permutations of $A(n,2)\setminus B(n,2)$ are precisely the permutations whose indexes are of the form 01100\ldots, and the number of such permutations is precisely $\UD{n}{011}=2{n\choose 3}-{n\choose 2}+1$ \cite{Shevelev} when $n\ge 2$, and from this the formula for $b(n,2)$ follows.  We will use similar ideas to compute $b(n,3)$.

We next consider Conjecture~\ref{C-Main} when $d$ is large. Observe that we have $b(n,d)=p(n,d)=0$ if $d>\floor{(n-1)/2}$.  Thus the largest value of $d$ such that Conjecture~\ref{C-Main} is non-trivial is $d=\floor{(n-1)/2}$, and in this case the conjecture does indeed hold.  To this end, let $EC(n)=2E(2n,n-1)$ denote the Eulerian-Catalan numbers, which have been studied recently by Bidkhori and Sullivant \cite{EulCat}.
\begin{thm}\label{T-n/2}
	For all $n\ge 0$,
	\[
		b(2n+1,n)=p(2n+1,n)=EC(n).
	\]
	Moreover, there exists an explicit bijection between $B(2n+1,n)$ and $P(2n+1,n)$.
\end{thm}

\begin{thm}\label{T-2n}
	For all $n\ge 1$,
	\[b(2n,n-1)=p(2n,n-1)=\half \sum_{k\ge 1,\ k\tr{ odd}} {2n\choose k} EC\l(\f{k-1}{2}\r)EC\l(\f{2n-k-1}{2}\r).\]
\end{thm}

Lastly, we provide a formula for $p(2n+1,n-1)$, which we predict also holds for $b(2n+1,n-1)$.
\begin{prop}\label{P-secondLast}
	For all $n\ge 2$, $p(2n+1,n-1)=$ \[2E(2n,n-2)+\rec{6}\sum_{\substack{1\le k\le 2n-1,\\ 1\le \ell\le 2n+1-k,\\ k,\ell\tr{ odd}}} {2n+1\choose k}{2n+1-k\choose \ell}EC\l(\f{k-1}{2}\r)EC\l(\f{\ell-1}{2}\r)EC\l(\f{2n-k-\ell}{2}\r).\]
\end{prop}

We collect some notation that will be used throughout the text.  Let $[n]=\{1,2,\ldots,n\}$.  If $\mf{c}$ is a cycle, we let $|\mf{c}|$ denote its length.  We will say that $\mf{c}$ is mostly increasing if $M(\mf{c})=\cD(\mf{c})$, or equivalently if 
$\cA(\mf{c})\ge \cD(\mf{c})$. We say that $\mf{c}$ is mostly decreasing if $M(\mf{c})=\cA(\mf{c})$.  We note that if $|\mf{c}|$ is odd, then $\mf{c}$ is either mostly increasing or mostly decreasing, but not both.  We let $C(n,d)$ denote the set of $n$-cycles of $S_n$ which have $M(\mf{c})=d$, and we let $c(n,d)=|C(n,d)|$.

\section{Proof of Theorem~\ref{T-1}}
We first prove Proposition~\ref{P-AscStart}, from which the formula for $b(n,1)$ will follow.  It will be of use to define $V(n,d):=E(n,d)\setminus A(n,d)$.  That is, $V(n,d)$ consists of all the permutations of size $n$ with $d$ descents which begin with a descent.  We let $v(n,d)=|V(n,d)|$.  We recall the following recurrence for Eulerian numbers, which is valid for all $n,d\ge 1$ \cite{Wolf}.

\begin{align}\label{E-Eulerian}
E(n,d)=(d+1)E(n-1,d)+(n-d)E(n-1,d-1).
\end{align}
\begin{proof}[Proof of Proposition~\ref{P-AscStart}]
	The result is certainly true for $d=0$, so assume that we have proven the result up to $d\ge 1$.  For any fixed $d$ the result is true for $n=1$, so assume the result has been proven up to $n\ge 2$.
	
	Define the map $\phi:A(n,d)\to S_{n-1}$ by sending $\pi \in A(n,d)$ to the word obtained by removing the letter $n$ from $\pi$.  We wish to determine the image of $\phi$.  Let $\pi \in A(n,d)$, and let $i$ denote the position of $n$ in $\pi$.  If $i=n$ or $\pi_{i-1}>\pi_{i+1}$ with $i>2$, then $\phi(\pi)$ will continue to have $d$ descents and begin with an ascent, so we will have $\phi(\pi)\in A(n-1,d)$.  If $i=2$ and $\pi_1>\pi_3$, then we will have $\phi(\pi)\in V(n-1,d)$.  If $\pi_{i-1}<\pi_{i+1}$ then we will have $\phi(\pi)\in A(n-1,d-1)$.  
	
	It remains to show how many times each element of the image is mapped to.  If $\pi\in A(n-1,d)$, then $n$ can be inserted in $\pi$ in $d+1$ ways to obtain an element of $A(n,d)$ (it can be placed at the end of $\pi$ or in between any $\pi_{i}>\pi_{i+1}$).  If $\pi \in A(n-1,d-1)$, then $n$ can be inserted in $\pi$ in $n-d$ ways to obtain an element of $A(n,d)$ (it can be placed in between any $\pi_i<\pi_{i+1}$).  If $\pi\in V(n-1,d)$, then $n$ must be inserted in between $\pi_1>\pi_2$ in order to have the word begin with an ascent.  With this and the inductive hypothesis, we conclude that
	\begin{align}
	a(n,d)&=(k+1)a(n-1,d)+(n-k)a(n-1,d-1)+v(n-1,d)\nonumber\\ 
	&=(k+1)^2E(n-2,d)+(n-k)kE(n-2,d-1)+v(n-1,d).\label{E-AEul}
	\end{align}
	Again by the inductive hypothesis and \eqref{E-Eulerian}, we have \[v(n-1,d)=E(n-1,d)-a(n-1,d)=E(n-1,d)-(d+1)E(n-2,d)=(n-k)E(n-2,d-1).\]
	Substituting this into \eqref{E-AEul} and applying \eqref{E-Eulerian} again gives the result.
\end{proof}

With this we can prove our first main result.

\begin{proof}[Proof of Theorem~\ref{T-1}]
	We have $B(n,1)=A(n,1)$, so the formula for $b(n,1)$ follows from Proposition~\ref{P-AscStart}.  It remains to construct a bijection $\phi$ between $B(n,1)$ and $P(n,1)$. 
	
	Any $\pi \in B(n,1)$ can be written as $\pi=xdy'$, where $d$ is the unique descent of $\pi$ and $x,y'$ are words containing only ascents.  Note that $x$ does not contain any letter $d'>d$, as otherwise if $\pi_j=d$ we would have $\pi_{j-1}>d$, which would imply that $\pi$ has at least two descents.  Thus we can write $y'=yz$ where $z=(d+1)\cdots n$ and $y$ contains only ascents.  For example, if $\sig=125783469$ we have $x=1257,\ d=8,\ y=346,\ z=9$. Note that $x$ is always non-empty (otherwise $\pi$ would not be ballot) and $y$ is always non-empty (otherwise $d$ would not be a descent), and this latter statement is equivalent to saying $x\ne 12\cdots(d-1)$.
	
	We first define our map $\phi$ for the permutations $\pi$ which have $1$ appearing in $x$, such as the permutation $\sig$ given above.  We will call a word $a=a_1\cdots a_r$ a consecutive run if $a_{i+1}=a_i+1$ for all $1\le i<r$.  We rewrite $xd$ as $x_1\cdots x_{k+1}$, where each $x_i$ is a maximal consecutive run.  For example, if $\sig=125783469$ we have $xd=12578=x_1x_2x_3$ with $x_1=12,\ x_2=5,\ x_3=78$.  Since we assumed that $xd$ contains 1 and is not equal to $12\cdots d$, we have that $xd$ is not itself a consecutive run, and thus we always have $k\ge 1$.
	
	We now rewrite $y$ as $y_1\cdots y_k$, where $y_i$ denotes the consecutive run consisting of all the elements that are larger than every element of $x_i$ and smaller than every element of $x_{i+1}$.  For example, if $\sig=125783469$ we have $y=346=y_1y_2$ with $y_1=34,\ y_2=6$.  Note that each $y_i$ is non-empty, as otherwise $x_ix_{i+1}$ would be a consecutive run, contradicting the maximality of $x_i$ and $x_{i+1}$. 
	
	Let $x_i'$ and $y_i'$ be the largest values of $x_i$ and $y_i$.  Note that $x'_1<y'_1<\cdots<x'_k<y'_k<x'_{k+1}$.  We define, for all $\pi$ with 1 in $x$, $\phi(\pi)=(x'_1,y'_1,\ldots,x'_k,y'_k,x_{k+1}')$.  We note that this is an element of $P(n,1)$ since $\phi(\pi)$ consists of a single non-trivial cycle on $2k+1$ elements which has exactly one cyclic descent.  For example, if $\sig=125783469$ we have $\phi(\sig)=(2,4,5,6,8)$.
	
	It remains to define $\phi$ for the case that $\pi$ has 1 in $y$.  If $\pi=xdyz$ as in the notation above, let $\pi'=ydxz$, noting that $\pi'\in B(n,1)$ since $x$ and $y$ are non-empty.  If $\pi$ has 1 in $y$, we define $\phi(\pi)=\widetilde{\phi(\pi')}$, where $\widetilde{\tau}$ denotes $\tau$ with all of its cycles reversed.  For example, if $\sig=125783469$ we have $\sigma'=346812579$ and $\phi(\sig')=\widetilde{\phi(\sig)}=\widetilde{(2,4,5,6,8)}=(8,6,5,4,2)$.  In this case we again have $\phi(\pi)\in P(n,1)$, so $\phi$ is indeed a map from $B(n,1)$ to $P(n,1)$.
	
	We claim that $\phi$ is invertible.  Namely, let $\pi\in P(n,1)$ be such that its non-trivial cycle $\mf{c}$ is mostly increasing, say $(\mf{c})=(x'_1,y'_1,\ldots,x'_k,y'_k,x'_{k+1})$ with these values increasing.  Let $y_i$ be the consecutive run starting at $x'_i+1$ and ending with $y'_i$, let $x_i$ be the consecutive run starting at $y'_{i-1}+1$ and ending at $x'_i$ (where we let $y'_0=0$), and let $z$ consist of the consecutive run from $d+1$ to $n$.  We define $\psi(\pi)=x_1x_2\cdots x_{k+1}y_1y_2\cdots y_kz$, which is the unique preimage of $\pi$ under $\phi$.  For example, if $\sig=(2,4,5,6,8)\in P(9,1)$ then $\psi(\sig)=125783469$.  If the cycle of $\pi$ is mostly decreasing we define $\psi(\pi)=\psi(\widetilde{\pi})'$, with the operation $'$ defined as before, and again one can verify that this sends $\pi$ to its unique preimage under $\phi$.  We conclude that $\phi$ and $\psi$ are inverses of each other, and hence that $\phi$ is a bijection.
\end{proof}

\section{Formulas for $p(n,d)$}
In order to find formulas for $p(n,d)$, we first require the following lemma.  Recall that $c(n,d)$ denotes the number of $n$-cycles $\pi$ of $S_n$ with $M(\pi)=d$.
\begin{lem}\label{L-Cyc}
	We have $c(1,0)=1$.  For $n\ge 3$ odd, $c(n,d)=0$ if $d>(n-1)/2$, and otherwise $c(n,d)=2E(n-1,d-1)$.
\end{lem}
\begin{proof}
	The result for $c(1,0)$ is immediate.  Let $\mf{c}$ denote an $n$-cycle.  Then $\cA(\mf{c})+\cD(\mf{c})=n$, and since $n$ is odd, one of these values is at most $(n-1)/2$.  We conclude that $M(\mf{c})\le (n-1)/2$ for all $\mf{c}$, and hence $c(n,d)=0$ if $d>(n-1)/2$.
	
	Now assume $d\le (n-1)/2$ with $n\ge 3$ odd.  Let $S(n,d)$ denote the permutations of $S_n$ which have exactly $d$ descents, and let $C^+(n,d)$ denote the cycles of $C(n,d)$ which are mostly increasing.  If we have $\pi\in S(n-1,d-1)$, define $\phi(\pi)=(\pi_1,\ldots,\pi_{n-1},n)$. Note that $\phi(\pi)$ has exactly one more cyclic ascent than $\pi$ has ascents, and similarly with regards to descents.  We conclude that $M(\phi(\pi))=\cD(\phi(\pi))=d$ since $d\le (n-1)/2$, so $\phi(\pi)\in C^+(n,d)$.  It is not too difficult to see that $\phi$ is a bijection onto $C^+(n,d)$ and that $|C^+(n,d)|=\half c(n,d)$.  Since $|S(n-1,d-1)|=E(n-1,d-1)$, we conclude the result.
\end{proof}

With this lemma we can prove the following recurrence relation for $p(n,d)$.

\begin{prop}\label{P-ORec}
	Let $p(0,0)=0$.  Then for all $n\ge 1$ and $d\ge 0$, we have that $p(n+1,d)-p(n,d)$ is equal to
	\begin{align*} \sum_{d'=1}^d&\l(\sum_{\substack{k\ge 0\\  k\tr{ even}}}2{n\choose k}E(k,d'-1)p(n-k,d-d')-\sum_{\substack{d'-1\le k\le 2(d'-1)\\ k\tr{ even}}}2{n\choose k}E(k,d'-1)p(n-k,d-d')\r).\end{align*}
\end{prop}

\begin{proof}
	Let $\phi:P(n+1,d)\to \bigcup_{1\le i\le n} S_i$ be the map defined by having $\phi(\pi)$ be the permutation obtained by first deleting the cycle containing $n+1$, and then relabeling the smallest remaining element with 1, the second smallest remaining element with 2, and so on.  For example, if $\pi=(256)(4)(137)$ then $\phi(\pi)=(134)(2)$.
	
	Let $\pi \in P(n+1,d)$ and let $\bar{c}$ denote the cycle of $\pi$ containing $n+1$.  If $|\bar{c}|=2\ell+1$ and $M(\bar{c})=d'$, then it is not difficult to see that $\phi(\pi)\in P(n-2\ell,d-d')$.  It is also not difficult to see that the number of times a given element of $P(n-2\ell,d-d')$ is mapped to is exactly ${n\choose 2\ell}c(n,d)$ (one first chooses the other $2\ell$ elements of the cycle containing $n+1$, and then one can arrange these $2\ell+1$ elements in $c(n,d)$ different ways).  We conclude that
	\[
	p(n+1,d)=\sum_{d'=0}^d\sum_{\substack{k\ge 0\\ k\tr{ even}}}{n\choose k}c(k+1,d')p(n-k,d-d').
	\]
	
	In order to get this in the form as stated, we observe that $c(k+1,0)=0$ unless $k=0$, in which case $c(1,0)=1$.  We also have by Lemma~\ref{L-Cyc} that $c(k+1,d')=0$ whenever $d'>k/2$, and that otherwise it is equal to $2E(k,d'-1)$ when $d'\ge 1$.  Thus,
	\[
	p(n+1,d)=p(n,d)+\sum_{d'=1}^d\sum_{\substack{k\ge 2d'\\ k\tr{ even}}}2{n\choose k}E(k,d'-1)p(n-k,d-d').
	\]
	
	In order to have the sum range over all even $k$, we simply add and subtract all of our missing terms, and we observe that $E(k,d'-1)=0$ for $k<d'-1$ (we can not use $k<d'$ since we have $E(0,1-1)=1$).
\end{proof}

We will need the following lemma in order to properly apply Proposition~\ref{P-ORec}.
\begin{lem}\label{L-Sum}
	Let $n,r,s$ be integers with $n\ge r+s$, and let $c,d$ be real numbers. Then
	\begin{align*}
	2\sum_{k\ge 0} {n\choose k} {n-k\choose r}{k\choose s}c^{n-k}d^k&=2{n\choose r}{n-r\choose s}c^rd^s(c+d)^{n-r-s},\\ 
	2\sum_{\substack{k\ge 0,\\ k\tr{ even}}} {n\choose k} {n-k\choose r}{k\choose s}c^{n-k}d^k&= {n\choose r}{n-r\choose s}c^rd^s((c+d)^{n-r-s}+(-1)^s(c-d)^{n-r-s}).
	\end{align*}
\end{lem}
\begin{proof}
	Consider $f(x,y)=\f{2}{r!s!}(cx+dy)^n$.  Then \[\f{\pa^{r+s}}{\pa x^r\pa y^s} f(x,y)|_{x=y=1}=2{n\choose r}{n-r\choose s}c^rd^s(c+d)^{n-r-s}.\]  But by the binomial theorem this is also equal to \[\f{\pa^{r+s}}{\pa x^r\pa y^s}\f{2}{r!s!}\sum_{k\ge 0} (cx)^{n-k}(dy)^k|_{x=y=1}=2\sum_{k\ge 0}{n-k\choose r}{k\choose s}c^{n-k}d^k.\]
	
	For the second part, let $S_e$ denote the sum of $\f{2}{r!s!}{n\choose k} {n-k\choose r}{k\choose s}c^{n-k}d^k$ over all even $k$ and $S_o$ the sum over all odd $k$.  By the first part we have that
	\[
	S_e-S_o=2\sum_{k\ge 0} {n\choose k} {n-k\choose r}{k\choose s}c^{n-k}(-d)^k=2{n\choose r}{n-r\choose s}c^rd^s(-1)^s(c-d)^{n-r-s}.
	\]
	Also by the first part we have $S_e+S_o=2{n\choose r}{n-r\choose s}c^rd^s(c+d)^{n-r-s}$.  We conclude the result by adding these two equations and dividing by 2.
\end{proof}

We are now ready to prove our formulas for $p(n,d)$.
\begin{proof}[Proof of Theorem~\ref{T-2}]
	As noted in the introduction, we have $b(n,2)=a(n,2)-|A(n,2)\setminus B(n,2)|$, and the elements in this last set are precisely the permutations counted by $\UD{n}{011}$.  When $n\ge 2$ this quantity is equal to $2{n\choose 3}-{n\choose 2}+1$, and $a(n,2)=3E(n-1,2)$ by Proposition~\ref{P-AscStart}.  We conclude that $b(n,2)$ satisfies this formula, so it remains to show that this is also the case for $p(n,2)$.
	
	The proposed formula equals $0=p(2,2)$ when $n=2$, so assume we have proven that $p(n,2)$ agrees with the formula up to and including $n\ge 2$.  By Proposition~\ref{P-ORec}, we have that
	\begin{align*}
	p(n+1,2)-p(n,2)&=P_1+P_2-S,
	\end{align*}
	where 
	
	\begin{align}
	P_i&=2\sum_{k\ge 0,\ k\tr{ even}}{n\choose k}E(k,i-1)p(n-k,2-i)\nonumber,\\ 
	S&=2{n\choose 0}E(0,0)p(n,1)+2{n\choose 2}E(2,1)p(n-2,0)=2(2^n-2n)+2{n\choose 2},\label{E-P2S}
	\end{align}
	where we used that $n\ge2$ in order to apply our formulas to $p(n,1)$ and $p(n-2,0)$.  
	
	We first evaluate $P_1$.  We have $E(k,0)=1$ for all $k$ and $p(n-k,1)=2^{n-k}-2(n-k)$ for all $k\ne n$ by Theorem~\ref{T-1}.  We wish to replace $p(n-k,1)$ in the terms of $P_1$ with $2^{n-k}-2(n-k)$ for all $k$.  This will be valid unless $n$ is even, in which case we will have added an extra term of $2{n\choose 0}(2^0-2\cdot 0)=2$ to the sum.  By subtracting off this term when $n$ is even, we find 
	\[P_1=2\sum_{k\ge 0,\ k\tr{ even}}{n\choose k}\l(2^{n-k}-2(n-k)\r)-2\cdot \half(1+(-1)^n).\]
	By applying Lemma~\ref{L-Sum}, first with $r=s=0,\ c=2,\ d=1$ and then with $r=c=d=1$ and $s=0$, we find that
	\begin{equation}
	P_1=(3^n+1)-2n2^{n-1}-1-(-1)^n.\label{E-P21}
	\end{equation}
	
	To evaluate $P_2$, we recall from \eqref{E1} that $E(k,1)=2^k-k-1$ and that $p(n-k,0)=1$ for all $k$.  Thus
	\[
	P_2=2\sum_{k\ge 0,\ k\tr{ even}}{n\choose k}(2^k-k-1).
	\]
	By applying Lemma~\ref{L-Sum} with the appropriate conditions, we find
	\[
		P_2=(3^n+(-1)^n)-n2^{n-1}-2^{n}.
	\]
	
	Adding this to \eqref{E-P21} and \eqref{E-P2S} gives
	\[
		p(n+1)-p(n)=2\cdot 3^n-(3n+6)2^{n-1}-2{n\choose 2}+4n.
	\]
	One can verify that $3^n-3n2^{n-1}-2{n\choose 3}+4{n\choose 2}-1$, which is equal to $3E(n-1,2)-2{n\choose 3}+{n\choose 2}-1$ by \eqref{E2}, also satisfies this recurrence, and since this and $p(n,2)$ agree at $n=2$, we conclude that the two functions must be equal to one another.
\end{proof}

We prove the formula for $p(n,3)$ in essentially the same way.
\begin{prop}\label{P-P3}
	For all $n\ge 4$, we have
	\[
	p(n,3)=4E(n-1,3)-\l({n\choose 3}-{n\choose 2}+4\r)2^{n-2}-22{n\choose 5}+16 {n\choose 4}-4{n\choose 3}+2n.
	\]
\end{prop}
\begin{proof}
	One can verify the statement holds for $n=4$, and from now on we assume $n\ge 4$.  By Proposition~\ref{P-ORec}, we have
	\[
		p(n+1)-p(n)=P_1+P_2+P_3-S,
	\]
	where
	\begin{align}
		P_i&=2\sum_{k\ge 0,\ k\tr{ even}}{n\choose k}E(k,i-1)p(n-k,3-i),\nonumber\\ 
		S&=2{n\choose 0}E(0,0)p(n,2)+2{n\choose 2}E(2,1)p(n-2,1)+2{n\choose 4}E(4,2)p(n-4,0)\nonumber \\
		&=2\l(3^n-3n2^{n-1}-2{n\choose 3}+4{n\choose 2}-1\r)+2{n\choose 2}(2^{n-2}-2(n-2))+22{n\choose 4}\label{E-P3S},
	\end{align}
	where we used that $E(2,2)=0$ to ignore the term with $k=2,d'=2$, and that $n\ge 4$ in order to use the formulas provided by Theorems~\ref{T-1} and \ref{T-2} and that $p(n-4,0)=1$.
	
	We now wish to evaluate each of the remaining three sums.  We first turn our attention to $P_1$.  We have $E(k,0)=1$ for all $k$.  If $f(n):=3^{n}-3n2^{n-1}-2{n\choose 3}+4{n\choose 2}-1$, then by Theorem~\ref{T-2} we have $p(n-k,2)=f(n-k)$ provided $n-k\ge 2$.  One can verify that in fact $p(n,2)=f(n)$ for all $n\ge 0$ except $f(1)=-1$.  Thus if we wish to replace each $p(n-k,2)$ term in $P_1$ with $f(n-k)$, we must subtract $2\cdot \half(1-(-1)^n){n\choose 1}(-1)$ from the expression (to deal with the $p(1)$ term if $n$ is odd), in total giving
	\[
		P_1=2\sum_{k\ge 0,\ k\tr{ even}}{n\choose k}\l(3^{n-k}-3(n-k)2^{n-k-1}-2{n-k\choose 3}+4{n-k\choose 2}-1\r)+(1-(-1)^n)n.
	\]
	By repeatedly applying Lemma~\ref{L-Sum} one concludes that
	\begin{equation}\label{E-P31}
		P_1=(4^n+2^n)-3n(3^{n-1}+1)-{n\choose 3}2^{n-2}+{n\choose 2}2^n-2^n+(1-(-1)^n)n.
	\end{equation}

	We next consider $P_2$.  We have $E(k,1)=2^k-k-1$ for all $k$, and if $g(n):=2^n-2n$ we have $p(n-k,1)=g(n-k)$ for all $k$ except that $g(0)=1$.  We thus have $P_2$ equal to
	\[
		2\sum_{k\ge 0,\ k\tr{ even}}{n\choose k}\l(2^n-2(n-k)2^k-k2^{n-k}+2k(n-k)-2^{n-k}+2(n-k)\r)-(1+(-1)^n)(2^n-n-1).
	\]
	Again applying Lemma~\ref{L-Sum} one finds that
	\begin{equation}\label{E-P32}
		P_2=4^n-2n(3^{n-1}+(-1)^{n-1})-n(3^{n-1}-1)+n(n-1)2^{n-1}-(3^n+1)+n2^n-(1+(-1)^n)(2^n-n-1).
	\end{equation}
	Finally we turn our attention to $P_3$.  We have $p(n-k,0)=1$ and $E(k,2)=3^k-(k+1)2^k+{k+1\choose 2}$ for all $k$.  By rewriting ${k+1\choose 2}={k\choose 2}+k$, we get
	\[
		P_3=2\sum_{k\ge 0,\ k\tr{ even}}3^k-k2^k-2^k+{k\choose 2}+k.
	\]
	Applying Lemma~\ref{L-Sum} gives
	\begin{equation}\label{E-P33}
	P_3=(4^n+(-2)^n)-2n(3^{n-1}-(-1)^{n-1})-(3^n+(-1)^n)+{n\choose 2}2^{n-2}+n2^{n-1}.
	\end{equation}
	
	Adding \eqref{E-P31}, \eqref{E-P32}, and \eqref{E-P33} and subtracting \eqref{E-P3S} gives an explicit recurrence relation for $p(n,3)$.  One can verify that the proposed formula for $p(n,3)$ also satisfies this recurrence, and since these functions agree at $n=4$ we conclude that the two functions are equal.
\end{proof}
In principle one could continue to use these sort of methods to compute $p(n,d)$ for any fixed $d$, though the computations would become somewhat involved.

\section{The Formula for $b(n,3)$.}
For $d=1$ and $2$, we found a formula for $b(n,d)$ by first finding a formula for $|A(n,d)\setminus B(n,d)|$.  This will also be our approach for $d=3$, though the situation is somewhat more complicated. Our main task will be to find a formula for the number of permutations whose index begins with 0110 and which have an additional descent somewhere else.  To this end, we introduce the following notation.

For any binary string $r$ ending in 0, let $f(r,n)$ denote the number of permutations $\pi\in S_n$ which have $r^\pi=rr'$, where $r'$ is a binary string with exactly one 1.  Given a binary string $r$, we let $r_\ell^-$ denote $r$ with its $\ell$th letter removed, and we let $r_\ell^1$ denote $r$ with its $\ell$th letter replaced by 1.  We will say that $r$ has a peak at position $i$ if $r_i=1$ and either $i=1$ or $r_{i-1}=0$, and we write $K(r)$ to denote the set of peaks of $r$.

\begin{prop}\label{P-BRec}
	Let $r$ be a binary string of size $k$ ending in a 0 and assume $n\ge 2$.  If $1\notin K(r)$, then 
	\[f(r,n+1)=2f(r,n)+(n-k)\UD{n}{r}+\UD{n}{r_{k}^1}+\sum_{i\in K(r)}f(r_{i-1}^-,n)+f(r_i^-,n),\]
	and if $1\in K(r)$,
	\[f(r,n+1)=2f(r,n)+(n-k)\UD{n}{r}+\UD{n}{r_{k}^1}+f(r_1^-,n)+\sum_{i\in K(r),\ i\ne 1}f(r_{i-1}^-,n)+f(r_i^-,n).\]
\end{prop}
\begin{proof}
	First assume $1\notin K(r)$.  Let $F(r,n)$ denote the set of permutations enumerated by $f(r,n)$, and let $G(r,n)$ denote the set of permutations enumerated by $\UD{n}{r}$.  Consider the map $\phi:F(r,n+1)\to S_n$ defined by having $\phi(\pi)$ be $\pi$ after removing the letter $n+1$.  Let $\pi\in F(r,n+1)$, and let $i$ denote the position of $n+1$ in $\pi$.
	
	If $i=n+1$, then it is not hard to see that $\phi(\pi)\in F(r,n)$.  If $k<i<n+1$, then $i$ is the unique descent of $\pi$ that is larger than $k$.  If $i>k+1$, then either $\pi_{i-1}>\pi_{i+1}$ and $\phi(\pi)\in F(r,n)$, or $\pi_{i-1}<\pi_{i+1}$ and $\phi(\pi)\in G(r,n)$.  If $i=k+1$, then either $\phi(\pi)\in G(r,n)$ if $\pi_{i-1}<\pi_{i+1}$, or $\phi(\pi)\in G(r_k^1,n)$ if $\pi_{i-1}>\pi_{i+1}$.
	
	Note that $i\ne 1$ since we assumed $1\notin K(r)$.  If $1<i\le k$, then we must have $i\in K(r)$ since $\pi_{i-1}<n+1>\pi_{i+1}$.  If $\pi_{i-1}<\pi_{i+1}$ then $\phi(\pi)\in F(r_i^-,n)$, and if $\pi_{i-1}>\pi_{i+1}$ then $\phi(\pi)\in F(r_{i-1}^-,n)$.  Thus we can restrict our codomain to $C:=F(r,n)\cup G(r,n)\cup G(r_k^1,n)\bigcup_{i\in K(r)}F(r_{i-1}^-,n)\cup F(r_i^-,n)$.
	
	Let $\pi\in C$.  We wish to deduce how many ways we can insert $n+1$ into $\pi$ and produce an element of $F(r,n+1)$.  If $\pi \in F(r,n)$, then $n+1$ can (only) be inserted either at the end of $\pi$ or right before the unique descent appearing after position $k$.  If $\pi \in G(r,n)$, then $\pi$ can (only) be placed before any of the $n-k$ letters whose position is at least $k+1$.  If $\pi \in G(r_k^1,n)$, then $\pi$ can (only) be inserted directly after position $k$.  If $\pi \in F(r_{i-1}^-,n)\cup F(r_i^-,n)$, then $n+1$ can (only) be inserted after position $i-1$.  With this we conclude the result when $1\notin K(r)$.
	
	If $1\in K(r)$, then essentially the same argument holds except one must include $G(r_1^-,n)$ in the codomain of $\phi$ (this coming from the case $i=1$).   Every element of $G(r_1^-,n)$ is mapped to exactly once by $\phi$ (namely by inserting $n+1$ at he beginning of the word), and from this we conclude the result.
\end{proof}

We will now iteratively apply Proposition~\ref{P-BRec} to determine $f(r,n)$ for various $r$.  In order to do so, we will need the following formulas from \cite{Shevelev}, which are valid provided $n-1\ge |r|$.
\begin{align*}
\UD{n}{0}=1,\hspace{2em} &\UD{n}{1}={n\choose 1}-1,\ \\ \UD{n}{01}={n\choose 2}-1,\hspace{2em} &\UD{n}{11}={n\choose 2}-{n\choose 1}+1,\\ \UD{n}{011}=2{n\choose 3}-{n\choose 2}+1,\hspace{2em} &\UD{n}{111}={n\choose 3}-{n\choose 2}+{n\choose 1}-1,\\ \UD{n}{0111}=3{n\choose 4}-&2{n\choose 3}+{n\choose 2}-1,\\ \UD{n}{01011}=16{n\choose 5}-5{n\choose 4}+{n\choose 2}-1,\hspace{2em} &\UD{n}{00111}=6{n\choose 5}-3{n\choose 4}+{n\choose 3}-1.
\end{align*}
We note that if $r'=r0$, then $\UD{n}{r'}=\UD{n}{r}$ provided $n\ge |r|>0$.

\begin{lem}
	For $n\ge 2$, $f(10,n)=(n-2)2^{n-1}-\half (3n-1)(n-2)$.
\end{lem}
\begin{proof}
	Let $n\ge 2$.  By Proposition~\ref{P-BRec} we know that $f(10,n+1)$ is equal to
	\begin{align*}
	2f(10,n)+(n-2)\UD{n}{10}+\UD{n}{11}+\UD{n}{0}=2f(10,n)+(n-2)(n-1)+{n\choose 2}-n+1+1,
	\end{align*}
	where we can replace $\UD{n}{r}$ with these values because $n+1\ge 2$.  We have $f(10,2)=0$ (as every permutation of $S_2$ has $|r^\pi|=1$).  One can verify that the proposed formula for $f(10,2)$ also satisfies these conditions, so the two functions must be equal to each other for $n\ge 2$.
\end{proof}
In essentially the same way we can prove the following set of results, whose details we omit.
\begin{lem}
	For $n\ge 2$, $f(00,n)=2^n-\half n^2-\f{3}{2}n+1$.
\end{lem}

\begin{lem}
	For $n\ge 2$, $f(110,n)=\rec{8}(n^2-5n+8)2^n-\f{2}{3}n^3+\f{7}{2}n^2-\f{35}{6}n+2$.
\end{lem}

\begin{lem}
	For $n\ge 2$, $f(010,n)=\rec{8}(n^2-n-8)2^n-\f{5}{6}n^3+3n^2-\rec{6}n-2$.
\end{lem}

\begin{lem}
	For $n\ge 3$, $f(0110,n)=({n\choose 2}-{n\choose 3}-4)2^{n-2}-11{n\choose 4}+5{n\choose 3}-2{n\choose 2}-2n+3$.
\end{lem}

We note that these results must be proved in (roughly) the stated order. This is because, for example, the recurrence for $f(0110,n)$ utilizes the formulas for $f(110,n)$ and $f(010,n)$.  With this we can prove our desired formula for $b(n,3)$.

\begin{proof}[Proof of Theorem~\ref{T-3}]
	Proposition~\ref{P-P3} shows that $p(n,3)$ satisfies the proposed formula, so it remains to deal with $b(n,3)$.  By Proposition~\ref{P-AscStart} we have
	\[
	b(n,3)=a(n,3)-|A(n,3)\setminus B(n,3)|=4E(n-1,3)-|A(n,3)\setminus B(n,3)|.
	\]
	
	If $\pi\in A(n,3)\setminus B(n,3)$, then it is not too difficult to see that $r^\pi$ must begin with $0111,\ 00111,\ 01011$, or $0110$.  In the first three cases $\pi$ is counted by one of $\UD{n}{0111},\ \UD{n}{00111}$, or $\UD{n}{01011}$, and in the last case it is counted by $f(0110,n)$.  For $n\ge 4$ we can write these $\UD{n}{r}$ values in terms of their polynomial expressions, and summing these four values gives
	\[
	|A(n,3)\setminus B(n,3)|=\l({n\choose 3}-{n\choose 2}+4\r)2^{n-2}+22{n\choose 5}-16 {n\choose 4}+4{n\choose 3}-2n,
	\]
	proving the result.
\end{proof}

We suspect that similar methods could be used to compute $b(n,d)$ for any fixed $d$, though the computations would become somewhat involved.  In the appendix we provide an alternative method that can be used to find formulas for $b(n,d)$ that does not involve $\UD{n}{r}$.

\section{Formulas for large $d$}
Recall that $|\mf{c}|$ denotes the length of the cycle $\mf{c}$.

\begin{lem}\label{L-MLim}
	If $\pi=\mf{c}_1\cdots \mf{c}_k \in P(n)$, then
	\[
		M(\pi)\le \f{n-k}{2},
	\]
	with equality if and only if $M(\mf{c}_i)=\f{|\mf{c}_i|-1}{2}$ for all $i$.
\end{lem}
\begin{proof}
	As noted in the proof of Lemma~\ref{L-Cyc}, if $\mf{c}$ is a cycle of odd length then $M(\mf{c})\le \f{|\mf{c}|-1}{2}$.  The result follows by applying this inequality to each $\mf{c}_i$ and noting that $\sum \mf{c}_i=n$.
\end{proof}
With Lemma~\ref{L-MLim} we can compute formulas for $p(n,d)$ when $d$ is large.
\begin{proof}[Proof of Theorem~\ref{T-n/2}]
	By Lemma~\ref{L-MLim}, we have $\pi\in P(2n+1,n)$ if and only if $\pi$ is a $(2n+1)$-cycle $\mf{c}$ with $M(\mf{c})=n$.  Thus Lemma~\ref{L-Cyc} implies that \[p(2n+1,n)=c(2n+1,n)=2E(2n,n-1)=EC(n).\]  It remains to establish a bijection from $B(2n+1,n)$ to $P(2n+1,n)$.
	
	If $\pi=\pi_1\cdots \pi_{2n+1} \in B(2n+1,n)$, let $\phi(\pi)=(\pi_1,\pi_2,\ldots,\pi_{2n+1})$.  Since $\pi$ contained $n$ descents, $\pi_1\pi_2\cdots \pi_{2n+1}\pi_1$ contains exactly $n$ or $n+1$ descents, and hence $M(\phi(\pi))=n$ and the codomain of this map is $P(2n+1,n)$.  The fact that this map is invertible is implicitly proven in the second proof of Theorem 1.1 of \cite{EulCat}.  Explicitly (using the notation of \cite{EulCat}), it is shown that if $w=(w_1,\ldots,w_{2n+1})$ has $n$ or $n+1$ cyclic descents (i.e. if $w\in P(2n+1,n)$), then there exist $n+1$ choices of $i$ such that $w_iw_{i+1}\cdots w_{2n+1}w_1\cdots w_{i-1}$ has $n$ descents, and exactly one of these choices for $i$ makes this word have exceedance 0 (i.e. makes the word be ballot).  We thus define $\psi(w)=w_iw_{i+1}\cdots w_{2n+1}w_1\cdots w_{i-1}$ with $i$ the unique value such that this word has $n$ descents and is ballot.  Then $\psi$ is the inverse of $\phi$, and hence these maps are bijections.
\end{proof}

We continue to apply Lemma~\ref{L-MLim} to compute formulas for $p(n,d)$.
\begin{prop}\label{P-p2n}
	\[p(2n,n-1)=\half \sum_{k\tr{ odd}} {2n\choose k} EC\l(\f{k-1}{2}\r)EC\l(\f{2n-k-1}{2}\r)\]
\end{prop}
\begin{proof}
	Since $2n$ is even, any $\pi \in P(2n)$ is the product of at least two odd cycles.  By Lemma~\ref{L-MLim} we have that $\pi \in P(2n,n-1)$ if and only if $\pi=\mf{c}\mf{d}$ with $\mf{c},\mf{d}$ odd cycles such that $M(\mf{c})=(|\mf{c}|-1)/2$ and $M(\mf{d})=(|\mf{d}|-1)/2$.  
	
	Consider the following procedure for generating an element $\pi\in P(2n,n-1)$.  Choose $k$ elements to be in the first cycle of $\pi$ (which also determines the elements of the second cycle), and then arrange the elements of these two cycles in $c(k,(k-1)/2)$ and $c(2n-k,(2n-k-1)/2)$ ways, respectively, so that the cycles have $M$ values $(k-1)/2$ and $(2n-k-1)/2$ respectively. By Lemma~\ref{L-Cyc} and the fact that we defined $EC(\ell)=2E(2\ell,\ell-1)$,  we conclude that $c(k,(k-1)/2)=EC((k-1)/2)$ and $c(2n-k,(2n-k-1)/2)=EC((2n-k-1)/2)$.  Putting these results together (and noting that this procedure double counts the elements of $P(2n,n-1)$ since we implicitly ordered the two cycles) gives the desired formula.
\end{proof}

\begin{proof}[Proof of Proposition~\ref{P-secondLast}]
	By Lemma~\ref{L-Cyc}, there are exactly $c(2n+1,n-1)=2E(2n,n-2)$ elements $\pi \in P(2n+1,n-1)$ that consist of a single cycle, so it remains to count the elements of $P(2n+1,n-1)$ that are not of this form.
	
	If $\pi$ is not a single cycle then, since $2n+1$ is odd, $\pi$ must be the product of at least 3 odd cycles.  By Lemma~\ref{L-MLim}, we must have $\pi =\mf{c}_1\mf{c}_2\mf{c}_3$ with $M(\mf{c}_i)=\f{|\mf{c}_i|-1}{2}$ for $i=1,2,3$.  We can construct such a $\pi$ by choosing $k$ elements (with $k<2n+1$ odd) to go into the first cycle of $\pi$, $\ell$ of the remaining $2n+1-k$ elements to go into the second cycle (which determines the elements of the third cycle), and then arranging the elements of each cycle.  As argued in the proof of Proposition~\ref{P-p2n}, there will be $EC((k-1)/2)$ ways to arrange the first cycle, $EC((\ell-1)/2)$ ways to arrange the second, and $EC((2n-k-\ell)/2)$ ways to arrange the third cycle.  This argument overcounts the elements of $P(2n+1,n-1)$ by a factor of 6 since we have implicitly placed an order on the cycles.  Putting all these results together gives the desired formula.
\end{proof}
In principle one can generalize these methods to compute $p(n,\floor{(n-1)/2}-d)$ for any finite $d$, though the computations would be somewhat tedious.

We now wish to find a formula for $b(2n,n-1)$, and to do so we introduce an additional statistic.  Given any $\pi\in S_n$ and $0\le k\le n-1$, let $T_k(\pi)=\sum_{i=1}^k \alt{q_i^\pi}$, with the convention that $T_0(\pi)=0$.  Define $T(\pi)=\min_{0\le k\le n-1}\{T_k(\pi)\}$.  We let $S(n,d,t)$ denote the set of permutations of $S_n$ with exactly $d$ descents and with $T(\pi)=t$, and we let $s(n,d,t)=|S(n,d,t)|$.  Note that $S(n,d,0)=B(n,d)$.  We further define $\pi^r:=\pi_n\pi_{n-1}\cdots \pi_2\pi_1$.

\begin{lem}
	If $\pi\in S(n,d,t)$, then $\pi^r\in S(n,n-1-d,t+2d-n+1)$.  
\end{lem}
\begin{proof}
	Observe that $q_i^{\pi^r}=-q_{n-i}^\pi$.  In particular this implies that $\D(\pi^r)=n-1-\D(\pi)=n-1-d$.  Define $R_k(\pi)=\sum_{i=k}^{n-1} -\alt{q_i^\pi}$ with the convention that $R_n(\pi)=0$, and let $R(\pi)=\min_{1\le k\le n}\{R_k(\pi)\}$.  Observe that $T_k(\pi^r)=R_{n-k}(\pi)$ for all $0\le k\le n-1$, and hence $T(\pi^r)=R(\pi)$.  
	
	Let $k$ and $\ell$ be the smallest integers such that $T(\pi)=T_k(\pi)$ and $R(\pi)=R_\ell(\pi)$.  We claim that $k=\ell-1$. Indeed, assume $k>\ell-1$.  By the minimality of $k$, we must have \[0<T_{\ell-1}(\pi) -T_k(\pi)=\sum_{i=\ell}^{k} -\alt{q^\pi_i}=\sum_{i=\ell}^{n-1} -\alt{q^\pi_i}+\sum_{i=k+1}^{n-1}\alt{q^\pi_i}=R_\ell(\pi)-R_{k+1}(\pi),\] a contradiction to $\ell$ being such that $R_\ell$ is minimal.  Similarly, if $k<\ell-1$ we have \[0<R_{k+1}(\pi)-R_\ell(\pi)=T_k(\pi)-T_{\ell-1}(\pi),\] a contradiction, so we conclude that $k=\ell-1$. 
	
	With this we have \[T(\pi)-R(\pi)=T_{\ell-1}(\pi)-R_{\ell}(\pi)=\sum_{i=0}^{n-1} \alt{q_i^\pi}=\A(\pi)-\D(\pi).\] 
	Since $R(\pi)=T(\pi^r)$, we conclude that \[T(\pi^r)=T(\pi)+\D(\pi)-\A(\pi)=t+d-(n-1-d)=t+2d-n+1\]
	as desired.
\end{proof}
\begin{cor}\label{C-reverse}
	$s(n,d,t)=s(n,n-1-d,t+2d-n+1)$.  In particular, $s(2n,n-1,0)=s(2n,n,-1)$.
\end{cor}
\begin{proof}
	By the previous lemma, the map $\phi:S_n\to S_n$ defined by $\phi(\pi)=\pi^r$ is an involution sending $S(n,d,t)$ to $S(n,n-1-d,t+2d-n+1)$ and vice versa, so $\phi$ is a bijection between these two sets.
\end{proof}

\begin{proof}[Proof of Theorem~\ref{T-2n}]
	We already know $p(2n,n-1)$ satisfies this formula by Proposition~\ref{P-p2n}, so it remains to prove that this is the case for $b(2n,n-1)$.  Let $w$ be a word composed of $k$ distinct positive integers.  We will say that $w$ is a Dyck word if $\sum_1^\ell \alt{q_i^w}\ge 0$ for all $1\le \ell\le k-1$ and if $\sum_1^{k-1} \alt{q_i^w}=0$.  Observe that $w$ being a Dyck word implies that $k$ is odd.
	
	We generate a permutation $\pi$ as follows.  Given an odd number $k$, choose a subset $S\sub [2n]$ of cardinality $k$.  Choose an ordering of the elements of $S$ in such a way that the resulting word $w_1$ is a Dyck word, and similarly choose an ordering of $[2n]\setminus S$ to get a Dyck word $w_2$.  The procedure then outputs $\pi=w_1w_2$. 
	
	Let $\pi=w_1w_2$ be a permutation generated by this procedure such that $w_1$ has length $k$. Then $w_1$ has $(k-1)/2$ descents and $w_2$ has $(2n-k-1)/2$ descents, so $\D(\pi)=n-1$ if $q_k^\pi=1$ and otherwise $\D(\pi)=n$.  Since $w_1$ is a Dyck word we have $T_m(\pi)\ge 0$ if $m<k$ and $T_k(\pi)=\alt{q_k^\pi}$.  Since $w_2$ is a Dyck word, we also have $T_m(\pi)=\alt{q_k^\pi}+\sum_{k+1}^m \alt{q_i^\pi}\ge \alt{q_k^\pi}$ for all $m>k+1$.  Thus $T(\pi)=0$ if $q_k^\pi=1$ and otherwise $T(\pi)=-1$.  We conclude that this procedure always generates an element of $S(2n,n-1,0)\cup S(2n,n,-1)$.  We claim that every permutation of $S(2n,n-1,0)\cup S(2n,n,-1)$ is generated in a unique way by this procedure.  
	
	Let $\pi\in S(2n,n-1,0)$, noting that $\sum_{1}^\ell \alt{q_i^\pi}\ge 0$ for all $\ell$ and that $\sum_1^{n-1} \alt{q_i^\pi}=1$.  Let $k$ denote the largest value such that $\sum_{i=1}^{k-1} \alt{q_i^\pi}=0$, where we allow for the case $k=1$.  Then $k$ is odd, $w_1:=\pi_1\cdots \pi_k$ is a Dyck word, and $w_2:=\pi_{k+1}\cdots \pi_n$ is also a Dyck word by the maximality of $k$, so $\pi=w_1w_2$ arises from this procedure.  Now assume that we can also write $\pi=w'_1w'_2$ with $w'_1,w'_2$ Dyck words and with $w'_1$ having length $\ell$. Note that $q_\ell^\pi=1$, as otherwise we would have $T(\pi)<0$.  If $\ell<k$ then \[\sum_1^{k-1} \alt{q_i^\pi}=\sum_1^{\ell-1} \alt{q_i^\pi}+\alt{q_\ell^\pi}+\sum_{\ell+1}^{k-1} \alt{q_i^\pi}=1+\sum_{\ell+1}^{k-1} \alt{q_i^\pi}\ge 1,\] a contradiction to the fact that $\sum_1^{k-1} \alt{q_i^\pi}=0$.  A symmetric argument shows that we must have $k=\ell$, and hence the decomposition $\pi=w_1w_2$ is unique.
	
	Now given $\pi\in S(2n,n,-1)$, let $k$ denote the smallest value such that $\sum_1^{k} \alt{q_i^\pi}=-1$ (note that this sum goes up to $k$ and not $k-1$ as in the previous case). By a similar argument as above, we find that $k$ is odd and that $\pi_1\cdots \pi_k$ and $\pi_{k+1}\cdots \pi_n$ are Dyck words, and moreover that this is the unique way to generate $\pi$ by this procedure.  Thus each element of $S(2n,n-1,0)\cup S(2n,n,-1)$ is generated uniquely by this procedure.  
	
	It is not too difficult to see that the number of Dyck words using the letters $\{a_1,\ldots,a_k\}$ with $k$ odd is precisely $b(k,(k-1)/2)$, which is equal to $EC((k-1)/2)$ by Theorem~\ref{T-n/2}. Thus the total number of ways to carry out this procedure is
	\[\sum_{k\ge 1,\ k\tr{ odd}} {2n\choose k} EC\l(\f{k-1}{2}\r)EC\l(\f{2n-k-1}{2}\r).\]
	Each way of carrying out this procedure produces a distinct element of $S(2n,n-1,0)\cup S(2n,n,-1)$, so we conclude that
	\[
		\sum_{k\ge 1,\ k\tr{ odd}} {2n\choose k} EC\l(\f{k-1}{2}\r)EC\l(\f{2n-k-1}{2}\r)=|S(2n,n-1,0)\cup S(2n,n,-1)|
	\]
	\[
	 =s(2n,n-1,0)+s(2n,n,-1)=2s(2n,n-1,0)=2b(2n,n-1),
	\]
	with the second to last equality coming from Corollary~\ref{C-reverse}.  We conclude the result.
\end{proof}

\section{Acknowledgments}
The author would like to thank Fan Chung for suggesting this research topic, as well as for her many helpful comments and suggestions.  The author would also like to thank a referee whose comments greatly simplified several proofs and formulas.

\bibliographystyle{plain}
\bibliography{OOP}

\newpage 
\section*{Appendix A: Formulas for $b(n,d)$}
We provide an alternative method for deriving formulas for $b(n,d)$.  Let $B(n,d,k)$ denote the subset of $B(n,d)$ which has $\pi_n=k$, and let $b(n,d,k)=|B(n,d,k)|$. Note that $B(n,d,k)=\emptyset$ whenever $d<0$.

\begin{lem}\label{L-fundpre}
	For $n\ge 2$, 
	\[b(n,d,k)=\sum_{k'\ge k} b(n-1,d-1,k')+\sum_{k'<k} b(n-1,d,k').\]
\end{lem}
\begin{proof}
	Given $\pi \in B(n,d,k)$, we define $\phi(\pi)\in S_{n-1}$ by having its $i$th letter $\phi(\pi)_i$ satisfy
	\[
	\phi(\pi)_i=\begin{cases*}
	\pi_i & $\pi_i<k$,\\ 
	\pi_i-1 & $\pi_i>k$.
	\end{cases*}
	\]
	Assume $\pi_{n-1}=j$.  If $j>k$ then $\phi(\pi)\in B(n-1,d-1,j-1)$, and if $j<k$ then $\phi(\pi)\in B(n-1,d,j)$.  Every element of $\bigcup_{k'\ge k} B(n-1,d-1,k')\cup \bigcup_{k'<k} B(n-1,d,k')$ is the image of a unique element of $B(n,d,k)$ under $\phi$, so $\phi$ is a bijection between these two sets and we conclude the result.
\end{proof}

\begin{lem}\label{L-tele}
	For $n\ge 2$, we have $b(n,d,1)=b(n,d-1,n)$.
\end{lem}
\begin{proof}
	By applying Lemma~\ref{L-fundpre} twice, we have $b(n,d,1)=\sum_{k=1}^{n-1} b(n-1,d-1,k)=b(n,d-1,n)$.
\end{proof}

\begin{lem}\label{L-fund}
	For $n,k\ge 2$, we have $b(n,d,k)=b(n,d,k-1)+b(n-1,d,k-1)-b(n-1,d-1,k-1)$.
\end{lem}
\begin{proof}
	This follows by considering $b(n,d,k)-b(n,d,k-1)$ and then applying Lemma~\ref{L-fundpre} to $b(n,d,k)$  and $b(n,d,k-1)$.
\end{proof}

\begin{prop}\label{P-Col1}
	For $n\ge2$,
	\[
	b(n,1,k)=\begin{cases*}
	2^{k-1} & $k\le n-2$,\\ 
	2^{n-2}-1 & $k=n-1$,\\ 
	2^{n-1}-2n+2 & $k=n$.
	\end{cases*}
	\]
\end{prop}
\begin{proof}
	We prove these formulas by double induction.  Note that all of these formulas hold for $n=2$.  Observe that $b(n,0,k)=0$ if $k<n$ and $b(n,0,n)=1$.  Thus $b(n,1,1)=1$ for all $n\ge 2$ by Lemma~\ref{L-tele}, which agrees with our proposed formula for $k=1$.  Inductively assume that we have verified the formula for $b(n,1,k')$ for all $n\ge 2$ and $1\le k'<k$, and then that we have inductively verified the formula for $b(n',1,k)$ for all $2\le n'<n$.  By Lemma~\ref{L-fund}, $b(n,1,k)$ is equal to
	\begin{equation}\label{Col1}
	b(n,1,k-1)+b(n-1,1,k-1)-b(n-1,0,k-1).
	\end{equation}
	If $k\le n-2$, inductively we know that \eqref{Col1} is equal to $2^{k-2}+2^{k-2}-0=2^{k-1}$.  If $k=n-1$, then \eqref{Col1} is equal to $2^{n-3}+(2^{n-3}-1)-0=2^{n-2}-1$.  If $k=n$, then \eqref{Col1} is equal to $(2^{n-2}-1)+(2^{n-2}-2n+4)-1=2^{n-1}-2n+2$.  We conclude by induction that these formulas hold.
\end{proof}

With this we can deduce the formula for $b(n,1)$.  Indeed, by Lemma~\ref{L-fundpre} we have
\[
b(n,1)=\sum_{k=1}^n b(n,1,k)=b(n+1,1,n+1)=2^n-2n.
\]

In a similar way one can prove the following, which can be used to derive the formula for $b(n,2)$.

\begin{prop}\label{P-Col2}
	For $n\ge 5$,
	\[
	b(n,2,k)=\begin{cases*}
	2^{n-k}3^{k-1}-(4n-k-3)2^{k-2} & $k\le n-4$,\\
	8\cdot 3^{n-4}-3n2^{n-5}-2& $k=n-3$,\\
	4\cdot 3^{n-3}-(3n-1)2^{n-4}-2n+7& $k=n-2$,\\
	2\cdot 3^{n-2}-(3n-2)2^{n-3}-2{n+1\choose 2}+8n-10& $k=n-1$,\\
	3^{n-1}-(3n-3)2^{n-2}-2{n+2\choose 3}+10{n+1\choose 2}-14n+5 & $k=n$.
	\end{cases*}
	\]
\end{prop}

\newpage
\section*{Appendix B: Computational Data}\label{A-Comp}
Below we have included some computational data for some of the statistics that we have considered.  The first table consists of values for $b(n,d)$.  We note that all of the values listed agree with the values for $p(n,d)$.

\begin{center}
	\begin{tabular}{|l||l|l|l|l|l|l|} \hline
		b(n,d) & d=0 & d=1 & d=2 & d=3 & d=4 & d=5 \\ \hline \hline
		n=1 & $1$ & $0$ & $0$ & $0$ & $0$ & $0$ \\ \hline
		n=2 & $1$ & $0$ & $0$ & $0$ & $0$ & $0$ \\ \hline
		n=3 & $1$ & $2$ & $0$ & $0$ & $0$ & $0$ \\ \hline
		n=4 & $1$ & $8$ & $0$ & $0$ & $0$ & $0$ \\ \hline
		n=5 & $1$ & $22$ & $22$ & $0$ & $0$ & $0$ \\ \hline
		n=6 & $1$ & $52$ & $172$ & $0$ & $0$ & $0$ \\ \hline
		n=7 & $1$ & $114$ & $856$ & $604$ & $0$ & $0$ \\ \hline
		n=8 & $1$ & $240$ & $3488$ & $7296$ & $0$ & $0$ \\ \hline
		n=9 & $1$ & $494$ & $12746$ & $54746$ & $31238$ & $0$ \\ \hline
		n=10 & $1$ & $1004$ & $43628$ & $330068$ & $518324$ & $0$ \\ \hline
		n=11 & $1$ & $2026$ & $143244$ & $1756878$ & $5300418$ & $2620708$ \\ \hline
		n=12 & $1$ & $4072$ & $457536$ & $8641800$ & $43235304$ & $55717312$ \\ \hline
		n=13 & $1$ & $8166$ & $1434318$ & $40298572$ & $309074508$ & $728888188$ \\ \hline
		n=14 & $1$ & $16356$ & $4438540$ & $180969752$ & $2026885824$ & $7589067592$ \\ \hline
		n=15 & $1$ & $32738$ & $13611136$ & $790697160$ & $12512691028$ & $69028576454$ \\ \hline
		n=16 & $1$ & $65504$ & $41473216$ & $3385019968$ & $73898171456$ & $573754927712$ \\ \hline
		n=17 & $1$ & $131038$ & $125797010$ & $14270283414$ & $422060869866$ & $4470473831914$ \\ \hline
		n=18 & $1$ & $262108$ & $380341580$ & $59457742524$ & $2349012559564$ & $33181419358420$ \\ \hline
		n=19 & $1$ & $524250$ & $1147318004$ & $245507935018$ & $12811010885886$ & $237191391335758$ \\ \hline
		n=20 & $1$ & $1048536$ & $3455325600$ & $1006678811272$ & $68751877461032$ & $1645761138814040$ \\ \hline
		n=21 & $1$ & $2097110$ & $10394291094$ & $4105447763032$ & $364232722279840$ & $11148787030131978$ \\ \hline
		n=22 & $1$ & $4194260$ & $31242645420$ & $16672235476128$ & $1909625025412472$ & $74065171862108524$ \\ \hline
		n=23 & $1$ & $8388562$ & $93853769320$ & $67482738851220$ & $9927594128105024$ & $484210423704506108$ \\ \hline
		n=24 & $1$ & $16777168$ & $281825553760$ & $272439143364672$ & $51256011278005824$ & $3123806527720851840$ \\ \hline
		n=25 & $1$ & $33554382$ & $846030314842$ & $1097660274098482$ & $263144690491841262$ & $19930831004237505532$ \\ \hline
	\end{tabular}
\end{center}

\~ \\ 

Below we have included tables for $b(n,k,d)$, along with row and column sums for each table.
\\ 
\\

\begin{center}

\begin{tabular}{|l||l|l|} \hline
	b(1,k,d) & d=0 & Row Sum: \\ \hline \hline
	k=1 & $1$ & $1$ \\ \hline
	Col Sum: & $1$ & $1$ \\ \hline
\end{tabular}
\end{center}
\~ \\ 

\begin{center}

\begin{tabular}{|l||l|l|} \hline
	b(2,k,d) & d=0 & Row Sum: \\ \hline \hline
	k=1 & $0$ & $0$ \\ \hline
	k=2 & $1$ & $1$ \\ \hline
	Col Sum: & $1$ & $1$ \\ \hline
\end{tabular}
\end{center}
\~ \\ 
 
\begin{center}
\begin{tabular}{|l||l|l|l|} \hline
	b(3,k,d) & d=0 & d=1 & Row Sum: \\ \hline \hline
	k=1 & $0$ & $1$ & $1$ \\ \hline
	k=2 & $0$ & $1$ & $1$ \\ \hline
	k=3 & $1$ & $0$ & $1$ \\ \hline
	Col Sum: & $1$ & $2$ & $3$ \\ \hline
\end{tabular}
\end{center}
\~ \\ 

\begin{center}

\begin{tabular}{|l||l|l|l|} \hline
	b(4,k,d) & d=0 & d=1 & Row Sum: \\ \hline \hline
	k=1 & $0$ & $1$ & $1$ \\ \hline
	k=2 & $0$ & $2$ & $2$ \\ \hline
	k=3 & $0$ & $3$ & $3$ \\ \hline
	k=4 & $1$ & $2$ & $3$ \\ \hline
	Col Sum: & $1$ & $8$ & $9$ \\ \hline
\end{tabular}
\end{center}
\~ \\ 

\begin{center}

\begin{tabular}{|l||l|l|l|l|} \hline
	b(5,k,d) & d=0 & d=1 & d=2 & Row Sum: \\ \hline \hline
	k=1 & $0$ & $1$ & $8$ & $9$ \\ \hline
	k=2 & $0$ & $2$ & $7$ & $9$ \\ \hline
	k=3 & $0$ & $4$ & $5$ & $9$ \\ \hline
	k=4 & $0$ & $7$ & $2$ & $9$ \\ \hline
	k=5 & $1$ & $8$ & $0$ & $9$ \\ \hline
	Col Sum: & $1$ & $22$ & $22$ & $45$ \\ \hline
\end{tabular}
\end{center}
\~ \\ 

\begin{center}

\begin{tabular}{|l||l|l|l|l|} \hline
	b(6,k,d) & d=0 & d=1 & d=2 & Row Sum: \\ \hline \hline
	k=1 & $0$ & $1$ & $22$ & $23$ \\ \hline
	k=2 & $0$ & $2$ & $29$ & $31$ \\ \hline
	k=3 & $0$ & $4$ & $34$ & $38$ \\ \hline
	k=4 & $0$ & $8$ & $35$ & $43$ \\ \hline
	k=5 & $0$ & $15$ & $30$ & $45$ \\ \hline
	k=6 & $1$ & $22$ & $22$ & $45$ \\ \hline
	Col Sum: & $1$ & $52$ & $172$ & $225$ \\ \hline
\end{tabular}
\end{center}
\~ \\ 

\begin{center}

\begin{tabular}{|l||l|l|l|l|l|} \hline
	b(7,k,d) & d=0 & d=1 & d=2 & d=3 & Row Sum: \\ \hline \hline
	k=1 & $0$ & $1$ & $52$ & $172$ & $225$ \\ \hline
	k=2 & $0$ & $2$ & $73$ & $150$ & $225$ \\ \hline
	k=3 & $0$ & $4$ & $100$ & $121$ & $225$ \\ \hline
	k=4 & $0$ & $8$ & $130$ & $87$ & $225$ \\ \hline
	k=5 & $0$ & $16$ & $157$ & $52$ & $225$ \\ \hline
	k=6 & $0$ & $31$ & $172$ & $22$ & $225$ \\ \hline
	k=7 & $1$ & $52$ & $172$ & $0$ & $225$ \\ \hline
	Col Sum: & $1$ & $114$ & $856$ & $604$ & $1575$ \\ \hline
\end{tabular}
\end{center}
\end{document}